\documentclass[graybox]{svmult}

\usepackage{mathptmx}       
\usepackage{helvet}         
\usepackage{courier}        
\usepackage{type1cm}        

\usepackage{makeidx}         
\usepackage{graphicx}        
\usepackage{multicol}        
\usepackage[bottom]{footmisc}
\usepackage[square,sort,numbers]{natbib}
\usepackage{url}

\usepackage{amssymb,bm}
\usepackage{amsbsy}
\usepackage{psfrag}

\renewcommand{\PackageWarningNoLine}[2]{}
\usepackage{amsmath}
\usepackage{amsfonts}

\font\tenscr=rsfs10  scaled 1200 
\font\sevenscr=rsfs10  scaled 650 
\font\fivescr=rsfs5  scaled 800 
\skewchar\tenscr='177 \skewchar\sevenscr='177 \skewchar\fivescr='177
\newfam\scrfam \textfont\scrfam=\tenscr \scriptfont\scrfam=\sevenscr
\scriptscriptfont\scrfam=\fivescr
\def\scr{\fam\scrfam}

\begin{document}
\def\su{{\scr {\scriptstyle U}}}
\def\sv{{\scr {\scriptstyle V}}}
\def\sw{{\scr {\scriptstyle W}}}
\def\suh{{\scr {\scriptstyle U}}_{\hspace{-1.5pt} \scriptscriptstyle h}}
\def\suj{{\scr {\scriptstyle U}}_{\hspace{-3pt} j}}
\def\suk{{\scr {\scriptstyle U}}_{\hspace{-2pt} k}}
\def\swj{{\scr {\scriptstyle W}}_{\hspace{-3pt} j}}
\def\HDivO{H(\mathrm{div},\O)}
\def\O{\Omega}
\def\p{\partial}
\def\R{\mathbb{R}}
\def\hu{\hat u}
\def\hsn{{\hat\sigma_n}}
\def\hsnj{{\hat\sigma_{n,j}}}
\def\hK{h_{\scriptscriptstyle K}}
\def\d{\displaystyle}
\def\tP{\tilde{P}}
\def\cN{\mathcal{N}}
\def\HHzBOh{H^{1/2}(\p\O_h)}
\def\HmHBOh{H^{-1/2}(\p\O_h)}
\def\HHzBOhj{H^{1/2}(\p\O_{j,h})}
\def\HmHBOhj{H^{-1/2}(\p\O_{j,h})}
\def\optn#1{\|#1\|_{\mathrm{opt},V}}
\def\KSum{\sum_{K\in\O_h}}
\def\R{\mathbb{R}}
\def\p{\partial}
\def\tzeta{\tilde{\zeta}}
\def\Hh{{H^{1/2}(\p K)}}
\def\Hmh{{H^{-1/2}(\p K)}}
\def\Hmhh{{H^{-1/2}(\p \hat K)}}
\def\HDiv{H(\mathrm{div};\,K)}
\def\HDivh{H(\mathrm{div};\,\hat K)}
\def\hK{h_{\scriptscriptstyle{K}}}
\def\cN{\mathcal{N}}
\def\div{\mathrm{div}\,}
\def\grad{\mathrm{grad}\,}
\def\hzeta{\hat\zeta}
\def\hbzeta{\bm{\hat\zeta}}
\def\hx{\hat x}
\def\hs{\hat s}
\def\hw{\hat w}
\def\Hhh{{H^{1/2}(\p\hat K)}}
\def\Hmhh{H^{-1/2}(\p\hat K)}
\def\phK{{\p\hat K}}
\def\hQ{{\hat Q}}
\def\hS{{\hat S}}
\newcommand{\hatK}{\hat{K}}
\newcommand{\tnorm}[1]{|\hspace{-0.025in}|\hspace{-0.025in}|#1|\hspace{-0.025in}|\hspace{-0.025in}|}

\title*{A one-level additive Schwarz preconditioner for a discontinuous Petrov-Galerkin method}
\titlerunning{One--level ASM for DPG}
\author{Andrew T. Barker\inst{1}\and
  Susanne C. Brenner\inst{1}\and
  Eun-Hee Park\inst{2}\and
  Li-Yeng Sung\inst{1}}
\institute{$^1$Department of Mathematics and Center for Computation and
 Technology, Louisiana State University, Baton Rouge, LA 70803,
 USA. \texttt{andrewb@math.lsu.edu, brenner@math.lsu.edu,sung@math.lsu.edu}
 \and $^2$Division of Computational Sciences in Mathematics,
 National Institute for Mathematical Sciences,
 Daejeon 305-811, South Korea. \texttt{eunheepark@nims.re.kr}}
\authorrunning{A.T. Barker, S.C. Brenner, E.-H. Park and L.-Y. Sung}
%
%
\maketitle

\section{A discontinuous Petrov-Galerkin method for a model Poisson problem}\label{sec:DPG}
 Discontinuous Petrov-Galerkin (DPG) methods are new discontinuous Galerkin methods
 \cite{DG:2010:DPGI,DG:2011:DPGII,DGN:2011:DPGIII,DG:2011:DPGIV,DG:2011:DPGAnalysis,GQ:2012:Practical}
 with interesting properties.
 In this article we
 consider a domain decomposition preconditioner for a DPG method for the Poisson problem.
\par
 Let $\O$ be a polyhedral domain in $\R^d$ ($d=2,3$), $\O_h$ be a simplicial triangulation of
 $\O$.
 Following the notation in \cite{GQ:2012:Practical},
 the model Poisson problem (in an ultraweak formulation) is to find $\su\in U$ such that
\begin{equation*}
  b(\su,\sv)=l(\sv)\qquad\forall\,\sv\in V,
\end{equation*}
 where  $U=[L_2(\O)]^d\times L_2(\O)\times H^{\frac12}_0(\p\O_h)\times H^{-\frac12}(\p\O_h)$,
 $V=H({\rm div};\O_h)\times H^1(\O_h)$,
\begin{align*}
  b(\su,\sv)&=\int_\O \sigma\cdot\tau\,dx-\sum_{K\in\O_h}\int_K u\,\div\tau\,dx
    +\sum_{K\in\O_h}\int_{\p K}\hat u\, \tau\cdot n\,ds\\
    &\hspace{40pt}-\sum_{K\in\O_h}\int_K\sigma\cdot\grad\,v\,dx
    +\sum_{K\in\O_h}\int_{\p K}v\,\hat\sigma_n\,ds
\end{align*}
 for $\su=(\sigma,u,\hat u,\hat\sigma_n)\in U$ and $\sv=(\tau,v)\in V$, and $l(\sv)=\int_\O fv\,dx$.
\par
 Here $H^{1/2}_0(\p\O_h)$ (resp. $H^{-1/2}(\p\O_h)$) is the subspace of $\prod_{K\in\O_h}H^{1/2}(\p K)$
 (resp. $\prod_{K\in\O_h}H^{-1/2}(\p K)$) consisting of the traces of functions in $H^1_0(\O)$ (resp. traces of the
 normal components of vector fields in
 $H(\mathrm{div};\O)$), and $H({\rm div};\O_h)$
 (resp. $H^1(\O_h)$) is the space of piecewise $H(\mathrm{div})$ vector fields
 (resp. $H^1$ functions).
 The inner product on $V$ is given by
\begin{equation*}
  \big((\tau_1,v_1),(\tau_2,v_2)\big)_V=\sum_{K\in\O_h}\int_K [\tau_1\cdot\tau_2
  +\div\tau_1\div\tau_2+v_1v_2+\grad v_1\cdot\grad v_2]\,dx.
\end{equation*}
\par
 The DPG method for the Poisson problem
 computes $\suh\in U_h $ such that
\begin{equation}\label{eq:DPG}
 b(\suh,\sv)=l(\sv)\qquad\forall\,\sv\in V_h.
\end{equation}
 Here the trial space $U_h \,(\subset U)$ is defined by
\begin{equation*}
  U_h=\prod_{K\in\O_h}[P_m(K)]^d\times \prod_{K\in\O_h}P_m(K)\times \tP_{m+1}(\p\O_h)\times P_m(\p\O_h),
\end{equation*}
 $P_m(K)$ is the space of polynomials of total degree $\leq m$ on
 an element $K$,
 $\tP_{m+1}(\p\O_h)=H^{1/2}_0(\p\O_h)\cap\prod_{K\in\O_h}\tP_{m+1}(\p K)$,
 where $\tP_{m+1}(\p K)$ is the restriction of $P_{m+1}(K)$ to $\p K$,
 and
 $P_m(\p\O_h)=H^{-1/2}(\p\O_h)\cap \prod_{K\in\O_h}P_m(\p K)$, where $P_m(\p K)$ is the space of piecewise
 polynomials on the faces of $K$ with total degree $\leq m$.
\par
 Let $V^r=\{ (\tau,v) \in V : \tau |_K \in [P_{m+2}(K)]^d, v|_K \in P_r(K) \;
 \forall\, K \in \Omega_h \}$ for some $r\geq m+d$.  The discrete trial-to-test map
  $ T_h:U_h\longrightarrow V^r $ is defined by
\begin{equation*}
  (T_h \suh,\sv)_{V} = b(\suh,\sv), \quad \forall \suh\in U_h,\;\sv \in V^r,
\end{equation*}
 and the test space $V_h$ is $T_hU_h$.
\par
 We can rewrite \eqref{eq:DPG} as
  $a_h(\suh,\sw)=l(T_h\sw)$
  for all $\sw\in U_h$,  where
  $$a_h(\su,\sw)=b_h(\su,T_h\sw)=(T_h\su,T_h\sw)_V$$
  is an SPD bilinear form on $V_h\times V_h$,
 and we define an operator $A_h:U_h\longrightarrow U_h'$  by
\begin{equation}\label{eq:AhDef}
  \langle A_h\su,\sw\rangle=a_h(\su,\sw)\qquad\forall\,\su,\sw\in U_h.
\end{equation}
  Our goal is to develop a one-level additive Schwarz
 preconditioner  for $A_h$ (cf. \cite{MN:1985:OneLevel}).
\par
 To avoid the proliferation of constants, we will use the notation
 $A\lesssim B$ (or $B\gtrsim A$) to represent the inequality $A\leq (\mathrm{constant})\times B$,
 where the positive constant only depends on the shape regularity of $\O_h$ and the polynomial
 degrees $m$ and $r$.  The notation $A\approx B$ is equivalent to $A\lesssim B$ and $B\lesssim A$.
\par
 A fundamental result in \cite{GQ:2012:Practical} is the equivalence
\begin{equation}\label{eq:Fundamental}
  a_h(\su,\su)
   \approx \|\sigma\|_{L_2(\O)}^2+\|u\|_{L_2(\O)}^2+\|\hu\|_{H^{1/2}(\p\O_h)}^2+
    \|\hsn\|_{H^{-1/2}(\p\O_h)}^2
\end{equation}
 that holds for all $\su=(\sigma,u,\hu,\hsn)\in U_h$, where
\begin{align}
  \|\hu\|_{H^{1/2}(\p\O_h)}^2&=\sum_{K\in\O_h}\|\hu\|_{H^{1/2}(\p K)}^2
  =\sum_{K\in\O_h}\inf_{w\in H^1(K),\, w|_{\p K}=\hu}\|w\|_{H^1(K)}^2,\label{eq:HHalf}\\
 \|\hsn\|_{H^{-1/2}(\p\O_h)}^2&=\sum_{K\in\O_h}\|\hsn\|_{H^{-1/2}(\p K)}^2
  =\sum_{K\in\O_h}\inf_{q\in \HDiv,\, q\cdot n|_{\p K}=\hsn}\|q\|_{\HDiv}^2.\label{eq:HMinusHalf}
\end{align}
 Therefore the analysis of domain decomposition preconditioners for $A_h$ requires a better
 understanding of the norms $\|\cdot\|_{H^{1/2}(\p K)}$ and $\|\cdot\|_{H^{-1/2}(\p K)}$
 on the discrete spaces $\tP_{m+1}(\p K)$ and $P_m(\p K)$.
%
\section{Explicit Expressions for the Norms on $\tP_{m+1}(\p K)$ and $P_{m}(\p K)$}\label{sec:Norms}
\begin{lemma}\label{lem:Hhh}
  We have
\begin{equation*}
 \|\tzeta\|_\Hh^2\approx \hK\Big(\|\tzeta\|_{L_2(\p K)}^2+\sum_{F\in \Sigma_K}|\tzeta|_{H^1(F)}^2\Big)
 \qquad\forall\,\tzeta\in \tP_{m+1}(\p K),
\end{equation*}
 where $\hK$ is the diameter of $K$ and $\Sigma_K$ is the set of the faces of $K$.
\end{lemma}
\begin{proof}
 Let $ \cN(K) $ be the set of nodal points of the $P_m$ Lagrange finite element associated with $K$
 and
 $\cN(\p K)$ be the set of points in $\cN(K)$ that are on $\p K$.
\par
 Given any $\tzeta\in\tP_{m+1}(\p K)$, we define $\tzeta_*\in P_{m+1}(K)$ by
\begin{equation}\label{eq:tzetaStarDef}
 \tzeta_*(p)=\begin{cases} \tzeta(p)&\qquad\text{if $p\in\cN(\p K)$},\\[4pt]
       \tzeta_{\p K}&\qquad\text{if $p\in\cN(K)\setminus\cN(\p K)$},
      \end{cases}
\end{equation}
 where $\tzeta_{\p K}$ is the mean value of $\tzeta$ over $\p  K$.  Since $\tzeta_*=\tzeta$ on $\p K$,
 we have
\begin{equation}\label{eq:tzetaStarFirstEstimate}
  \|\tzeta\|_\Hh=\inf_{w\in H^1(K), w|_{\p K}=\zeta}\|w\|_{H^1(K)}\leq \|\tzeta_*\|_{H^1(K)}.
\end{equation}
\par
 Suppose $w\in H^1(K)$ satisfies $w=\tzeta$ on $\p K$.
 It follows from \eqref{eq:tzetaStarDef} 
 and the trace
 theorem with scaling that
\begin{equation}\label{eq:tzetaStarL2Est}
  \|\tzeta_*\|_{L_2(K)}^2\lesssim \hK\|\zeta\|_{L_2(\p K)}^2=\hK\|w\|_{L_2(\p K)}^2
  \lesssim \|w\|_{H^1(K)}^2,
\end{equation}
 and, by standard estimates,
\begin{align}\label{eq:tzetaStarH1Est}
  |\tzeta_*|_{H^1(K)}^2=|\tzeta_*-\tzeta_{\p K}|_{H^1(K)}^2
       &\lesssim \hK^{-1}\|\tzeta_*-\tzeta_{\p K}\|_{L_2(\p K)}^2\notag\\
       &=\hK^{-1}\|w-w_{\p K}\|_{L_2(\p K)}^2\lesssim |w|_{H^1(K)}^2.
\end{align}
 Combining \eqref{eq:tzetaStarFirstEstimate}--\eqref{eq:tzetaStarH1Est}, we have
 $\|\tzeta\|_\Hh^2 \approx \|\tzeta_*\|_{H^1(K)}^2$.  The lemma then follows from
 \eqref{eq:tzetaStarDef},
 the equivalence of norms on finite dimensional spaces and scaling. \qed
\end{proof}
\begin{lemma}\label{lem:Hmhh}
 We have
\begin{equation*}
  \|\zeta\|_{\Hmh}^2\approx \hK\|\zeta\|_{L_2(\p K)}^2+\hK^{-d}\Big(\int_{\p K}\zeta ds\Big)^2
  \qquad\forall\,\zeta\in P_m(\p K).
\end{equation*}
\end{lemma}
\begin{proof}  We begin with the reference simplex $\hat K$.  Let $RT_m(\hat K)$ be the
 $m$-th order Raviart-Thomas space (cf. \cite{BF:1991:Book}).
 Given any
 $\zeta\in P_m(\p\hat K)$, we introduce a (nonempty) subspace
  $RT_m(\hat K,\zeta)=\{q\in RT_m(\hat K):\;q\cdot n= \zeta \;\text{on}\; \p\hat K\;\text{and}\;
   \div q\in P_0(\hat K)\}$ of $RT_m(\hat K)$.
\par
 Let $\zeta_*\in RT_m(\hat K,\zeta)$ be defined by
\begin{equation*}
  \zeta_*=\min_{q\in RT_m(\hat K,\zeta)}\|q\|_{L_2(\hat K)}.
\end{equation*}
 Then the map $\hS:P_m(\p\hat K)\longrightarrow  RT_m(\hat K)$ that maps $\zeta$ to $\zeta_*$ is linear
 and one-to-one, and we have $(\hS\zeta)\cdot n=\zeta$ on $\p\hat K$, $\div(\hS\zeta)\in P_0(\hat K)$ and
\begin{equation}\label{eq:hSL2}
  \|\hS\zeta\|_{L_2(\hat K)}\approx \|\zeta\|_{L_2(\phK)}\qquad\forall\,\zeta\in P_m(\phK).
\end{equation}
\par
 Let $\zeta_1,\ldots,\zeta_{N_m}$ be a basis of $P_m(\p\hat K)$ and
 $1=\phi_1,\ldots,\phi_{N_m}\in \Hhh$ satisfy
%
  $\det\Big[\int_{\phK} \zeta_i\phi_j\,d\hs\Big]_{1\leq i,j\leq N_m}\neq0$.
%
 We define the map $\hQ:\HDivh\longrightarrow P_m(\phK)$ by
\begin{equation*}
 \int_{\phK} (\hQ q)\phi_j\,d\hs=
 \langle q\cdot n,\phi_j\rangle_{\Hmhh\times\Hhh}\qquad\text{for}\quad 1\leq j\leq N_m.
\end{equation*}
 It follows from the definition of $\hQ$ that
%
  $\|\hQ q\|_{L_2(\phK)}\lesssim \|q\|_{\HDivh}$ for all $q\in\HDivh$,
%
 and $\hQ q=\zeta$ if $q\cdot n=\zeta\in P_m(\phK)$, in which case
\begin{equation}\label{eq:hSzetaL2}
  \|\hS\zeta\|_{L_2(\hat K)}\lesssim
   \|\zeta\|_{L_2(\p\hat K)}=\|\hQ q\|_{L_2(\phK)}\lesssim\|q\|_{\HDivh}.
\end{equation}
 Moreover, since $\phi_1=1$, we have
\begin{equation*}
  \int_{\hat K}\div(\hS\zeta)\,d\hx=\int_{\phK} (\hQ q)1d\hs
       =\langle q\cdot n,1\rangle_{\Hmhh\times\Hhh}=\int_{\hat K}\div q\,d\hx
\end{equation*}
 and hence
\begin{equation}\label{eq:hSzetaDiv}
  \|\div(\hS\zeta)\|_{L_2(\hat K)}\lesssim \|\div q\|_{L_2(\hat K)}.
\end{equation}
\par
 Now we turn to a general simplex $K$.  It follows from \eqref{eq:hSL2}--\eqref{eq:hSzetaDiv}
 and standard properties of the Piola transform for $H(\mathrm{div})$
 (cf. \cite{Monk:2003:Maxwell}) that there exists a linear map
 $S:P_m(\p K)\longrightarrow RT_m(K)$ with the following properties:
 \par\noindent
 (i) $(S\zeta)\cdot n=\zeta$
  and hence
\begin{equation*}
   \|\zeta\|_{\Hmh}=\inf_{q\in \HDiv,\, q\cdot n|_{\p K}=\zeta}\|q\|_{\HDiv}
   \leq \|S\zeta\|_{\HDiv}\qquad\forall\,\zeta\in P_m(\p K),
\end{equation*}
 (ii) for any $q\in \HDiv$ such that $q\cdot n=\zeta$, we have
\begin{equation*}
   \|S\zeta\|_{\HDiv}\lesssim \|q\|_{\HDiv},
\end{equation*}
 (iii) $\div(S\zeta)\in P_0(K)$ and hence
\begin{equation*}
  \int_K\div(S\zeta)\,dx=\int_{\p K}\zeta\,ds\quad\text{or}\quad
   \|\div(S\zeta)\|_{L_2(K)}^2=\Big(\int_{\p K}\zeta\,ds\Big)^2/|K|,
\end{equation*}
 (iv) we have
\begin{equation*}
  \hK^{-d}\|S\zeta\|_{L_2(K)}^2\approx \hK^{-(d-1)}\|\zeta\|_{L_2(\p K)}^2.
\end{equation*}
 Properties (i)--(iv) then imply
\begin{equation*}
  \|\zeta\|_{\Hmh}^2\approx \|S\zeta\|_{\HDiv}^2\approx \hK\|\zeta\|_{L_2(\p K)}^2+
     \hK^{-d}\Big(\int_{\p K}\zeta\,ds\Big)^2. \qquad\qed
\end{equation*}
\end{proof}
%
\section{A Domain Decomposition Preconditioner}\label{sec:DD}
 Let $\O$ be partitioned into overlapping subdomains $\O_1,\ldots,\O_J$ that are aligned with
 $\O_h$.  The overlap among the subdomains is measured by  $\delta$ and we assume (cf. \cite{TW:2005:DD})
 there is a partition of
 unity $\theta_1,\ldots,\theta_J \in C^\infty(\bar\O)$ that satisfies the usual
 properties: $ \theta_j \geq 0$, $ \sum_{j=1}^J \theta_j = 1 $ on $
 \bar\O $, $ \theta_j = 0 $ on $ \O \setminus \O_j $, and
\begin{equation}\label{eq:GradientBdd}
  \|\nabla\theta_j\|_{L_\infty(\O)}\lesssim \delta^{-1}\qquad\forall\,1\leq j\leq J.
\end{equation}
\par
 We take the subdomain space to be
  $U_j=\{\su\in U_h:\,\su=0\;\text{on $\O\setminus\O_j$}\}$.
 Let $\su=(\sigma,u,\hu,\hsn)\in U_h$.  Then $\su\in U_j$ if and
 only if (i) $\sigma$ and $u$ vanish on every $K$ outside $\O_j$ and (ii) $\hu$
 and $\hsn$ vanish on $\p K$ for every $K$ outside $\O_j$.
%
 We define
 $a_j(\cdot,\cdot)$ to be the restriction of $a_h(\cdot,\cdot)$ on $U_j\times U_j$.
 Let $A_j:U_j\longrightarrow U_j'$ be defined by
\begin{equation}\label{eq:AjDef}
 \langle A_j\suj,\swj\rangle=a_j(\suj,\swj)\qquad\forall\,\suj,\swj\in U_j.
\end{equation}
 It follows from \eqref{eq:Fundamental} that
\begin{equation}\label{eq:SubdomainEnergyNorm}
  a_j(\suj,\suj)\approx \|\sigma_j\|_{L_2(\O_j)}^2+\|u_j\|_{L_2(\O_j)}^2
  +\|\hu_j\|_{\HHzBOhj}^2+\|\hsnj\|_{\HmHBOhj}^2,
\end{equation}
 where $\suj=(\sigma_j,u_j,\hu_j,\hsnj)\in U_j$,
 $\O_{j,h}$ is the triangulation of $\O_j$ induced by $\O_h$ and the
 norms $\|\cdot\|_{\HHzBOhj}$ and $\|\cdot\|_{\HmHBOhj}$ are analogous
 to those in \eqref{eq:HHalf} and \eqref{eq:HMinusHalf}.
\par
 Let $I_j:U_j\longrightarrow U_h$ be the natural injection.  The one-level
 additive Schwarz preconditioner $B_h:U_h'\longrightarrow U_h$ is defined by
%
  $$B_h=\sum_{j=1}^J I_j A_j^{-1} I_j^t.$$
%
\goodbreak
\begin{lemma}\label{lem:LowerBdd}
 We have
  $$\lambda_{\min}(B_hA_h)\gtrsim \delta^2.$$
\end{lemma}
\begin{proof}
 Let $ I_{h,1}$, $I_{h,2}$, $I_{h,3}$ and $ I_{h,4} $ be the nodal interpolation operators
 for the components $\prod_{K\in\O_h}\big[P_m(K)\big]^d$, $\prod_{K\in\O_h}P_m(K)$, $\tP_{m+1}(\p\O_h)$
 and $P_m(\p \O_h)$ of $U_h$ respectively.
 Given any $\su=(\sigma,u,\hu,\hsn)\in U_h$, we define $\suj\in U_j$ by
\begin{equation*}
 \suj=\big(I_{h,1}(\theta_j\sigma),I_{h,2}(\theta_ju),I_{h,3}(\theta_j\hu),I_{h,4}(\theta_j\hsn)\big).
\end{equation*}
 Then we have
 $\su=\sum_{j=1}^J\suj$
 and, in view of \eqref{eq:AjDef} and \eqref{eq:SubdomainEnergyNorm}, \goodbreak
\begin{align}\label{eq:ujEnergy}
 \langle A_j\suj,\suj\rangle&\approx
  \|I_{h,1}(\theta_j\sigma)\|_{L_2(\O_j)}^2+
  \|I_{h,2}(\theta_j u)\|_{L_2(\O_j)}^2\notag\\
  &\hspace{50pt}+\|I_{h,3}(\theta_j\hu)\|_{\HHzBOhj}^2
  +\|I_{h,4}(\theta_j\hsn)\|_{\HmHBOhj}^2.
\end{align}
\par
 The following bounds for the first two terms on the right-hand side of \eqref{eq:ujEnergy} are
 straightforward:
\begin{equation}\label{eq:L2Bdd}
  \|I_{h,1}(\theta_j\sigma)\|_{L_2(\O_j)}^2\lesssim \|\sigma\|_{L_2(\O_j)}^2\quad\text{and}
  \quad \|I_{h,2}(\theta_j u)\|_{L_2(\O_j)}^2\lesssim  \|u\|_{L_2(\O_j)}^2.
\end{equation}
 We will use Lemma~\ref{lem:Hhh} and Lemma~\ref{lem:Hmhh} to derive the following bounds
\begin{align}
  \|I_{h,3}(\theta_j\hu)\|_{\HHzBOhj}^2&\lesssim \delta^{-2}\|\hu\|_{\HHzBOhj}^2\label{eq:Trace1},\\
 \|I_{h,4}(\theta_j\hsn)\|_{\HmHBOhj}^2&\lesssim \delta^{-2}\|\hsn\|_{\HmHBOhj}^2\label{eq:Trace2}.
\end{align}
\par
  Let $K\in \O_{j,h}$.  It follows from Lemma~\ref{lem:Hhh}, \eqref{eq:GradientBdd} and standard discrete estimates
   that
\begin{align*}
  &\|I_{h,3}(\theta_j\hat u)\|_{\Hh}^2\approx
   \hK\Big(\|I_{h,3}(\theta_j\hat u)\|_{L_2(\p K)}^2
   +\sum_{F\in \Sigma_K}|I_{h,3}(\theta_j\hat u)|_{H^1(F)}^2\Big)\\
   &\hspace{30pt}\lesssim \hK\|\hat u\|_{L_2(\p K)}^2+\hK\sum_{F\in \Sigma_K}\big(
    \|\nabla\theta_j\|_{L_\infty(\O)}^2\|\hat u\|_{L_2(F)}^2+\|\theta_j\|_{L_\infty(\O)}^2
    |\hat u|_{H^1(F)}^2\big)\\
   &\hspace{30pt}\lesssim \hK\|\hat u\|_{L_2(\p K)}^2+ \hK\delta^{-2}\|\hat u\|_{L_2(\p K)}^2+\hK
    \sum_{F\in \Sigma_K}|\hat u|_{H^1(F)}^2
   \lesssim \delta^{-2}\|\hat u\|_{\Hh}^2.
\end{align*}
 Summing up this estimate over all the simplexes in $\O_{j,h}$ yields \eqref{eq:Trace1}.
\par
 Similarly, it follows from Lemma~\ref{lem:Hmhh} and \eqref{eq:GradientBdd} that
\begin{align*}
  &\|I_{h,4}(\theta_j\hsn)\|_{\Hmhh}^2\approx
  \hK\|I_{h,4}(\theta_j\hsn)\|_{L_2(\p K)}^2+\hK^{-d}\Big(\int_{\p K}I_{h,4}(\theta_j\hsn)\, ds\Big)^2\\
  &\hspace{20pt}\lesssim \hK\|\hsn\|_{L_2(\p K)}^2+\hK^{-d}\Big(\int_{\p K}I_{h,4}
  \big[(\theta_j-\theta_j^K)\hsn\big]\,ds\Big)^2 +\hK^{-d}\theta_j^K
      \Big(\int_{\p K}\hsn\, ds\Big)^2\\
  &\hspace{20pt}\lesssim \hK\|\hsn\|_{L_2(\p K)}^2+\hK\delta^{-2}\|\hsn\|_{L_2(\p K)}^2
  +\hK^{-d}\Big(\int_{\p K}\hsn\, ds\Big)^2
  \lesssim \delta^{-2}\|\hsn\|_{\Hmh}^2,
\end{align*}
 where $\theta_j^K$ is the mean value of $\sigma_j$ over $K$.  Summing up this estimate over all the
 simplexes in $\O_{j,h}$ gives us \eqref{eq:Trace2}.
\par
  Putting \eqref{eq:AhDef}, \eqref{eq:Fundamental} and
 \eqref{eq:ujEnergy}--\eqref{eq:Trace2} together we find
%
 $\sum_{j=1}^J\langle A_j\suj,\suj\rangle\lesssim \delta^{-2}\langle A_h\su,\su\rangle$,
%
 which implies $\lambda_{\min}(B_hA_h)\gtrsim\delta^2$ by the standard theory of additive Schwarz
 preconditioners \cite{TW:2005:DD}.  \qed
\end{proof}
\par
 Combining  Lemma~\ref{lem:LowerBdd} with the standard estimate
 $\lambda_{\max}(B_hA_h)\lesssim 1$, we obtain the following theorem.
\goodbreak\newpage
\begin{theorem}\label{thm:CondNumEst}
  We have
  $$\kappa(B_hA_h)=\frac{\lambda_{\max}(B_hA_h)}{\lambda_{\min}(B_hA_h)}\leq C \delta^{-2},$$
 where the positive constant $C$ depends only on the shape regularity of $\O_h$ and the polynomial degrees
 $m$ and $r$.
\end{theorem}
\begin{remark}\label{rem:TensorProduct}
  Theorem~\ref{thm:CondNumEst} is also valid for DPG methods based on tensor product finite elements.
\end{remark}
%
\section{Numerical results}\label{sec:NumericalResults}
 We solve the Poisson problem on the square $(0,1)^2$
 with exact solution $ u = \sin(\pi x_1) \sin(\pi x_2) $ and uniform square meshes.
 The trial space is based on $Q_1$ polynomials for $\sigma$ and
 $u$, $P_2$ polynomials for $\hat u$, and $P_1$ polynomials for $\hsn$.  We use bicubic polynomials for
 the space $V^r$ in the construction of the trial-to-test map $T_h$.
%
%
\par
 The number of conjugate gradient iterations required to reduce the residual by $ 10^{10} $ are given in
 Table~\ref{iterationtable} for four overlapping subdomains.
 The linear growth of the number of iterations for
 the unpreconditioned system is consistent with the condition number estimate $\kappa(A_h)\lesssim h^{-2}$
 in \cite{GQ:2012:Practical}.  Note that in this case the boundary of every subdomain has a nonempty intersection
 with $\p\O$ and it is not difficult to use a discrete Poincar\'e inequality to show that the estimate
 in Theorem~\ref{thm:CondNumEst} can be improved to $\kappa(B_hA_h)\lesssim |\ln h|\delta^{-1}$.  This
 is consistent with the observed growth of the number of iterations for the preconditioned system as $\delta$
 decreases.
\begin{table}[h]
\caption{Number of iterations for the Schwarz preconditioner with subdomain size $ H = 1/2 $.}
\label{iterationtable}
\begin{center}\begin{tabular}{ll|ll}
\hline\noalign{\smallskip}
    $h$ &  $\delta$&  unpreconditioned& preconditioned \\
\noalign{\smallskip}\hline\noalign{\smallskip}
  $ 2^{-2} $& $ 2^{-2} $& 496& 14 \\
  $ 2^{-3} $& $ 2^{-3} $& 1556& 17 \\
        & $ 2^{-2} $& & 14 \\
  $ 2^{-4} $& $ 2^{-4} $& 3865& 20 \\
        & $ 2^{-3} $& & 17 \\
        & $ 2^{-2} $& & 14 \\
  $ 2^{-5} $& $ 2^{-5} $& 8793& 27 \\ 
         & $ 2^{-4} $& & 20 \\ 
         & $ 2^{-3} $& & 18 \\
\hline\noalign{\smallskip}
\end{tabular}\end{center}
\end{table}
\par
 In Table~\ref{scalingtable} we display the results for $h=2^{-5}$ and
 various subdomain sizes $H$ with $\delta=H/2$.  The estimate $\kappa(B_hA_h)\lesssim \delta^{-2}\approx
 H^{-2}$ is consistent with the observed linear growth of the number of iterations for the preconditioned system
 as $H$ decreases.
 Such a condition number estimate for the one-level additive Schwarz preconditioner
 is known to be sharp for standard finite element methods \cite{Brenner:2007:DD16}.
\begin{table}[hh]
\caption{Number of iterations with $ h = 2^{-5} $ and various subdomain sizes $ H $ with $ \delta = H/2 $.}
\label{scalingtable}
  \begin{center}\begin{tabular}{ll|ll}
      \hline\noalign{\smallskip}
    $h$ &  $H$&  unpreconditioned& preconditioned \\
      \noalign{\smallskip}\hline\noalign{\smallskip}
   $ 2^{-5} $& $ 2^{-1} $&  8793& 15 \\
             & $ 2^{-2} $&      & 25 \\
            & $ 2^{-3} $&       & 45 \\
            & $ 2^{-4} $&       & 89 \\
      \hline\noalign{\smallskip}
\end{tabular}\end{center}
\end{table}
\newpage
\begin{acknowledgement}
 The work of the first author was
 supported in part by the National Science Foundation VIGRE Grant
 DMS-07-39382. The work of the second and fourth authors was supported in
 part by the National Science Foundation under Grant No.
 DMS-10-16332.  The work of the third author was supported in part by a KRCF research
 fellowship for young scientists.  The authors would also like to thank Leszek Demkowicz for helpful
 discussions.
\end{acknowledgement}

\end{document}